\documentclass{birkjour}
\usepackage{amssymb, amsmath, amsthm, graphics, comment, xspace, enumerate}
\usepackage{graphicx}
 \newtheorem{thm}{Theorem}[section]
 \newtheorem{cor}[thm]{Corollary}
 \newtheorem{lem}[thm]{Lemma}
 \newtheorem{prop}[thm]{Proposition}
 \newtheorem{defn}[thm]{Definition}
 \newtheorem{rem}[thm]{Remark}

\begin{document}

%
%
%
%
%
%
%
%
%

\title[Hyperfunction  Semigroups]{Hyperfunction  Semigroups}

\author[Kosti\' c M.]{Kosti\'c Marko}

%


\author[Pilipovi\' c S.]{Pilipovi\' c Stevan}

\author[Velinov D.]{Velinov Daniel}
%


\begin{abstract}
We analyze Fourier hyperfunction and hyperfunction
semigroups with non-densely defined
generators and their connections with local convoluted
$C$-semigroups. Structural theorems and
spectral characterizations give necessary and sufficient conditions
for the existence of such semigroups generated by a closed not
necessarily densely defined operator $A$.
\end{abstract}

\maketitle
\section{Introduction and preliminaries}
\noindent The papers on ultradistribution semigroups, \cite{kps}, \cite{msd} extend the classical theory of semigroups, (see
\cite{li121}, \cite{cha}, \cite{ito1},
 \cite{ki90}, \cite{ko98} and
\cite{ku112}). S. \={O}uchi \cite{o192}
was the first who introduced the class of hyperfunction
semigroups, more general than that of distribution and
ultradistribution semigroups and in \cite{o193} he considered the
abstract Cauchy problem in the space of hyperfunctions. Furthermore, generators of
hyperfunction semigroups in the sense of \cite{o192} are not
necessarily densely defined. A.N. Kochubei, \cite{kochu} considered hyperfunction solutions
on abstract differential equations of higher order.
We analyze  Fourier hyperfunction
semigroups with non-densely defined generators continuing
over the investigations of Roumieu type ultradistribution semigroups
and constructed examples of tempered ultradistribution
semigroups \cite{kps} as well as of Fourier hyperfunction semigroups with
non-densely defined generators.
An analysis of R. Beals
 ~\cite[Theorem 2']{b41} gives an example of
a densely defined operator $A$ in the Hardy space $H^{2}({\mathbb
C}_{+})$ which generates a hyperfunction semigroup of \cite{o192}
but this operator is not a generator of any
ultradistribution semigroup, and any (local) integrated $C$-semigroups,
$C\in L(H^{2}({\mathbb C}_{+}))$.
Our main interest is the existence of fundamental solutions for the
Cauchy problems with initial data being hyperfunctions.

In the definition of  infinitesimal generators for distribution and
ultradistribution semigroups in the non-quasi-analytic case, all
authors use test functions supported by $[0,\infty).$ Such an
approach cannot be used in the case of Fourier hyperfunction
semigroups since  in the quasi-analytic case only the zero function
has this property. Because of that, we define such semigroups on
test spaces $\mathcal{P}_{*}$ and $\mathcal{P}_{*,a}$ ($a>0$) but
the axioms for such semigroups as well as the definition of
infinitesimal generator are given on subspaces  of quoted spaces
consisting of functions $\phi$ with the property $\phi(0)=0$ and
$\phi'(0)=0$. We note that the same  can be done for the
distribution and ultradistribution semigroups (we leave this for
another paper).

Section 2 is devoted to Fourier hyperfunction semigroups.
As we mentioned, the definition of such semigroups
is intrinsically different than that of ultradistribution
semigroups because test functions
with the support bounded on the left cannot be used.
Fourier hyperfunction semigroups with densely
defined infinitesimal generators
were introduced by Y. Ito \cite{ito2}
related to the
corresponding Cauchy problem \cite {ito1}.
We give structural and spectral characterizations
of Fourier- and exponentially bounded Fourier hyperfunction
semigroups with  non-dense infinitesimal generators,
their relations with the convoluted semigroups and
to the corresponding Cauchy problems. Spectral properties of
hyperfunction semigroups give a new insight to S. \={O}uchi's
results.

\subsection{Hyperfunction and Fourier hyperfunction type spaces}

The basic facts about  hyperfunctions and Fourier
hyperfunctions  of M. Sato can be found on an
elementary level in the monograph of A. Kaneko \cite{kan}
(see also \cite{mor}, \cite{hor},
\cite{kaw1}-\cite{kaw2}). Let $E$ be a Banach space,
$\Omega$ be an open set in $\mathbb{C}$ containing an open set
$I\subset \mathbb{R}$ as a closed subset and let ${\mathcal
O}(\Omega)$ be the space of $E-$valued holomorphic functions on
$\Omega$ endowed with the topology of uniform convergence on compact
sets of $\Omega$.
The space of $E-$valued hyperfunctions on $I$ is defined as
${\mathcal B}(I,E):={\mathcal O}(\Omega\setminus I,E)/{\mathcal
O}(\Omega,E).$ A representative of $f=[f(z)]\in {\mathcal B}(I,E)$,
$f\in{\mathcal O}(\Omega\setminus I,E)$ is called a defining
function of $f$.
The space of hyperfunctions
supported by a compact set $K\subset I$ with values in $E$ is
denoted by $\Gamma_K(I,{\mathcal B}(E))={\mathcal B}(K,E).$ It is
the space of continuous linear mapping from ${\mathcal A}(K)$ into
$E$, where ${\mathcal A}(K)$ is the inductive limit type space of analytic functions in
neighborhoods of $K$ endowed with the appropriate topology \cite{k78}.
Denote by ${\mathcal A}(\mathbb{R})$ the space of real
analytic functions on $\mathbb{R}$: ${\mathcal A}(\mathbb{R})=$proj
lim$_{K\subset\subset{\mathbb R}}{\mathcal A}(K)$. The space of
continuous linear mappings from ${\mathcal A}(\mathbb{R})$ into $E,$
denoted by ${\mathcal B}_c(\mathbb{R},E),$ is consisted of compactly
supported elements of ${\mathcal B}(K,E),$ where $K$ varies through
the family of all compact sets in ${\mathbb R}.$ We denote by
$\mathcal{B}_+(\mathbb{R},E)$ the space of $E-$valued hyperfunctions
whose supports are contained in $[0,\infty).$
As in the scalar case ($E=\mathbb{R}$) we have, if
$f\in {\mathcal B}_c({\mathbb R},E)$ and supp$f\subset
\{a\}$, then $f=\sum _{n=0}^{\infty}\delta^{(n)}(\cdot-a)x_n,
\; x_n\in E,$ where $\lim \limits_{n\rightarrow
\infty}(n!||x_n||)^{1/n}$ $= 0.$
Let $\mathbb{D}=\{-\infty,+\infty\}\cup{\mathbb R}$ be the radial
compactification of the space ${\mathbb R}.$
Put $I_\nu=(-1/\nu,1/\nu),\ \nu>0.$ For
$\delta>0$, the space ${\tilde{\mathcal O}}^{-\delta}({\mathbb
D}+iI_\nu)$ is defined as a subspace of $\mathcal{O}({\mathbb
R}+iI_\nu)$ with the property that for every $K \subset\subset
I_\nu$ and $ \varepsilon>0$ there exists a suitable $C>0$ such that
$ |F(z)| \leq Ce^{-(\delta-\varepsilon)|Re z|},\; z\in {\mathbb
R}+iK.$ Then $ {\mathcal P}_*(\mathbb{D}):=$ind$\lim_
{n\rightarrow \infty}\tilde{\mathcal O}^{-1/n}({\mathbb D}+iI_n)$ is
the space of all {\it rapidly decreasing, real analytic functions}
(cf. ~\cite[Definition 8.2.1]{kan}) and the space of
Fourier hyperfunctions ${\mathcal Q}(\mathbb{D},E)$ is the space of
continuous linear mappings from ${\mathcal P}_*(\mathbb{D})$ into
$E$ endowed with the strong topology. We point out that Fourier
hyperfunctions were firstly introduced by M. Sato in \cite{satovi}
who called them {\it slowly increasing hyperfunctions}. Let us note
that the sub-index $*$ in ${\mathcal P}_*(\mathbb{D})$ does not have
the meaning as in the case of ultradistributions. This is often used
notation in the literature (cf. \cite{kan}). Recall, the restriction mapping
$\mathcal{Q}(\mathbb{D},E) \rightarrow {\mathcal
B}(\mathbb{R},E)$ is surjective, see ~\cite[Theorem 8.4.1]{kan}. For further
relations between the spaces $\mathcal{B}(\mathbb{R})$ and
${\mathcal Q}(\mathbb{D}),$ we refer to ~\cite[Section 8]{kan}.

Recall \cite{kan}, an operator of the form
$P(d/dt)=\sum_{k=0}^{\infty}b_k(d/dt)^k$ is called a {\it local
operator} if $\lim \limits_{k\rightarrow \infty}
(|b_k|k!)^{1/k}=0.$ Note that the composition and the sum
of local operators is again a local operator.

The main structural property of ${\mathcal Q}(\mathbb{D})$ says that
every element $f\in {\mathcal Q}(\mathbb{D})$ is of the form
$f=P(d/dt)F,$ where $P$ is a local operator and $F$ is a continuous
slowly increasing function, that is, for every $\varepsilon>0$ there
exists $C_\varepsilon>0$ such that $|F(t)|\leq C_\varepsilon
e^{\varepsilon|t|},\; t \in \mathbb{R}.$ More precisely, we have the
following global structural theorem (cf. ~\cite[Proposition 8.1.6,
Lemma 8.1.7, Theorem 8.4.9]{kan}), reformulated here with a sequence
$(L_p)_p$:

{\it Let, formally,
\begin{equation}\label{str3}
P_{L_p}(d/dt)=\prod_{p=1}^\infty(1+\frac{L_p^2}{p^{2}}d^2/dt^2)=\sum_{p=0}^\infty
a_pd^p/dt^p,
\end{equation}
where $(L_p)_p$  is a sequence decreasing to $0.$
 This is a local operator and we call it hyperfunction operator.Then \cite{kan}:
 \\ \indent
Let $T\in {\mathcal Q}({\mathbb D},E)$.  There is a local
operator $P_{L_p}(-id/dt)$ (with a corresponding sequence $(L_p)_p$)
and a continuous slowly increasing function $f:{\mathbb
R}\rightarrow E$, which means that, for every $\varepsilon>0$ there exists
$C_\varepsilon>0$ such that $||f(x)||\leq C_\varepsilon
e^{\varepsilon|x|},\; x\in {\mathbb R}$ and that
$T=P_{L_p}(-id/dt)f$. }

If a hyperfunction is compactly supported, supp$f\subset K,$ $f\in
{\mathcal B}(K,E)$, then we have the above representation with a
corresponding local operator  $P_{L_p}(-id/dt)$ and a continuous
$E-$valued function in a neighborhood of $K$.

The spaces of Fourier hyperfunctions were also analyzed by J. Chung,
S.-Y. Chung and D. Kim in \cite{chungovi}-\cite{chungov}. Following
this approach, we have that ${\mathcal P}_*(\mathbb{D})$ is
(topologically) equal to the space of $C^{\infty}-$functions $\phi$
defined on ${\mathbb R}$ with the property:
$ (\exists h>0) (||\phi||_{h}<\infty),$
where the norms $||\cdot||_{h}, \;h>0,$ are defined by
$||\phi||_{h}:= \sup
\{{||\phi^{(n)}(x)||e^{|x|/h}}/{(h^n n!)} : n\in\mathbb{N}_0,\
x\in \mathbb{R}\},
$
equipped with the corresponding inductive limit topology when
$h\rightarrow +\infty$. The next lemma can be proved by the standard
arguments using the norms $||\phi||_{h}.$

\begin{lem}\label{djuka}
If $ \phi,\ \psi \in {\mathcal P}_*(\mathbb{D}),$ then
$\phi*_0\psi=\int_0^t\phi(\tau)\psi(t-\tau)\, d\tau\, , t>0$ is in ${\mathcal P}_*(\mathbb{D})$ and the mapping $\ast_{0
} : {\mathcal P}_*(\mathbb{D})\times{\mathcal
P}_*(\mathbb{D})\rightarrow{\mathcal P}_*(\mathbb{D})$ is
continuous.
\end{lem}
\begin{proof}
 Suppose
$x\in\mathbb{R},\ n\in\mathbb{N}$ and $h_1>0$ fulfill
$||\phi||_{h_1}<\infty.$ Suppose that $h>2 h_1$ satisfies
$||\psi||_{\frac{h}{2}}<\infty$ and put $h_2=\frac{hh_1}{h-h_1}.$
We will use the next inequality which holds for evey  $t$,$
\frac{|x|}{h} \leq \frac{|x-t|}{h}+\frac{|t|}{h}\leq
\frac{|x-t|}{h}+\frac{|t|}{h_1}-\frac{|t|}{h_2}.
$ We have
$$\sup_{n\in\mathbb{N}_0,\
x\in\mathbb{R}}\frac{e^{|x|/h}\bigl|\Bigl(\int_{0}^{x}
\phi(t)\psi(x-t)\, dt\Bigr)^{(n)}\bigr|}{h^nn!}\leq$$
$$\leq\sup_{n\in\mathbb{N}_0,\
x\in\mathbb{R}}\frac{e^{|x|/h}\int_{0}^{x}|\phi(t)\psi^{(n)}(x-t)|\, dt}{h^nn!}+
\sum_{j=0}^{n-1}\sup_{n\in\mathbb{N},\
x\in\mathbb{R}}\frac{e^{|x|/h}|\phi^{(j)}(x)||\psi^{(n-1-j)}(0)|}{h^nn!}=
$$
$$
=I+II.
$$
We will estimate separately $I$ and $II.$
$$I\leq \sup_{t\in \mathbb{R}}\Bigl(|\phi(t)|
e^{\frac{|t|}{h_1}}\Bigr)\Bigl(\int_{0}^{x}e^{-\frac{|t|}{h_2}}\, dt\Bigr)
\sup_{n\in\mathbb{N}_0,\ x,t\in \mathbb{R}}
\frac{|\psi^{(n)}(x-t)|e^{|x-t|/h}}{h^n n!},
$$
$$II\leq \frac{1}{2^n}\sum_{j=0}^{n-1}\sup_{j\in\mathbb{N}_0,\ x\in\mathbb{R}}
\frac{e^{|x|/h}|\phi^{(j)}(x)|}{(h/2)^jj!}
\sup_{n-j\in\mathbb{N}_0}\frac{|\psi^{(n-1-j)}(0)|}{(h/2)^{n-j}(n-j)!}
$$
This gives $\phi \ast_{0} \psi \in {\mathcal P}_*(\mathbb{D})$ while
the continuity of the mapping $\ast_{0} : {\mathcal
P}_*(\mathbb{D})\times{\mathcal P}_*(\mathbb{D})\rightarrow{\mathcal
P}_*(\mathbb{D})$ follows similarly. This completes the proof of the
lemma.
\end{proof}
Now we will transfer the definitions and assertions for Roumieu
tempered ultradistributions to Fourier hyperfunctions.
\begin{defn}\label{fhprostor}
Let $a\geq 0.$ Then
$$
{\mathcal{P}}_{*,a}({\mathbb D}):=\{\phi \in C^\infty({\mathbb R}) :
e^{a\cdot}\phi\in{\mathcal P}_*(\mathbb D)\}.
$$
Define the convergence in this space by
$$\phi_n\rightarrow 0 \mbox{ in } {\mathcal{P}}_{*,a}(\mathbb{D}) \mbox{ iff }
e^{a \cdot}\phi_n\rightarrow 0 \mbox{ in } {\mathcal P}_*({\mathbb
D}).
$$
We denote by $\mathcal{Q}_{a}(\mathbb{D},E)$ the space of continuous
linear mappings from $\mathcal{P}_{*,a}(\mathbb{D})$ into $E$
endowed with the strong topology.
\end{defn}
We have:
\begin{equation}\label{strnova}
F\in {\mathcal{Q}}_{a}(\mathbb{D},E) \mbox{ iff } e^{-a\cdot}F\in
{\mathcal{Q}}(\mathbb{D},E).
\end{equation}

\begin{prop}\label{novo}
Let $G \in \mathcal{Q}_{a}(\mathbb{D},L(E)).$ Then there exists a
local operator $P$  and a function $g\in C({\mathbb R},L(E))$ with
the property that for every $\varepsilon>0$ there exists
$C_\varepsilon>0$ such that
$$e^{-ax}||g(x)||\leq C_\varepsilon e^{\varepsilon|x|}, \;x\in \mathbb{R}\;
\mbox{ and }\;  G=P(d/dt)g.
$$
\end{prop}
\begin{proof}
From the structure theorem for the space ${\mathcal Q}({\mathbb D},L(E))$ and since $e^{-a\cdot}G\in{\mathcal Q}({\mathbb D},L(E))$, there exists a local operator $P$ and a function $g_1$ with the property that for every $\varepsilon>0$ there is corresponding $C_{\varepsilon}>0$ such that $$\|g_1(x)\|\leq C_{\varepsilon} e^{\varepsilon |x|},\, x\in{\mathbb R}\quad \mbox{and} \quad G=e^{ax}P(d/dt)g_1\, .$$ We put $g(x)=e^{ax}g_1(x)$, $x\in\mathbb R$. Using Leibnitz formula , we have $$e^{ax}P(d/dt)g_1(x)=\sum\limits_{t=0}^{\infty}(\sum\limits_{k=0}^{\infty}{{t+k}\choose t} (-1)^ka^kb_{k+t})(e^{ax}g_1(x))^{(t)}\, .$$
The assertion will be proved if we show that $\lim\limits_{|t|\rightarrow\infty}(|c_t|t!)^{\frac{1}{t}}=0$, where $c_t=\sum\limits_{k=0}^{\infty}{{k+t}\choose k}a^k b_{k+t}$. To prove this, we  use $${{t+k}\choose k}\leq (t+k)^k\leq 2^kk^k+2^kt^k\leq 2^k(k^k+k^ke^t)=2^kk^k(1+e^t)\, ,$$ where we used $t^k\leq k^ke^t$. The last inequality is clear for $k\geq t$. For $k<t$, we put $k=\nu t$. First let we note that $\nu\ln{\nu}\in (-1,0)$.
Then $\nu t\ln t\leq\nu t\ln t +\nu t\ln \nu+t$. Hence $t^k\leq k^ke^t$. Now, $$c_t=\sum\limits_{k=0}^{\infty} 2^kk^k(1+e^t)a^kb_{k+t}=\sum\limits_{k=0}^{\infty}(2a)^kk^k(1+e^t)b_{k+t}\, .$$ The coefficients $b_{k+t}$ are coefficients of a local operator, so for all $\varepsilon>0$ , exists $M\in\mathbb N$ such that for all $t+k>M$, $|b_{k+t}|(t+k)!<{\varepsilon^{t+k}}$. With this we have $$t!|c_t|\leq (1+e^t)
\sum\limits_{k=0}^{\infty}\frac{(2a)^kk^k(t+k)!t!|b_{k+t}|}{(t+k)!}\leq\sum\limits_{k=0}^{\infty} \frac{(2a)^k(1+e)^te^kk!t!(t+k)!|b_{t+k}|}{(t+k)!}\leq$$ $$\leq\sum\limits_{k=0}^{\infty}\frac{(2a)^k(1+e^t)k!t!(t+k)!|b_{t+k}|}{t!k!}\leq
\sum\limits_{k=0}^{\infty}{(2ae)^k(1+e^t)}{\varepsilon^{t+k}}=(1+e^t){\varepsilon^t}
\sum\limits_{k=0}^{\infty}{(2ae\varepsilon)}^k\, $$ and the assertion follows since we can choose $\varepsilon$ arbitrary small.
\end{proof}

\begin{rem} By Lemma \ref{djuka}, one can easily prove that, if
$ \phi,\ \psi \in {\mathcal P}_{*,a}(\mathbb{D}),$ then
$\phi*_0\psi\in {\mathcal P}_{*,a}(\mathbb{D})$ and the mapping $\ast_{0
} : {\mathcal P}_{*,a}(\mathbb{D})\times{\mathcal
P}_{\ast,a}(\mathbb{D})\rightarrow{\mathcal P}_{*,a}(\mathbb{D})$ is
continuous.\end{rem}

For the needs of the Laplace
transform we define the space $\mathcal{P}_*([-r,\infty]), r>0$.
Note that $[-r,\infty]$ is compact in $\mathbb D.$

$\mathcal{P}_{\ast}([-r,\infty],h)$ is defined as the  space of smooth functions
$\phi$ on $(-r,\infty)$ with the property
$
||\phi||_{\ast,-r,h}<\infty,\;
\mbox{where}$
$$
||\phi||_{\ast,-r,h}:= \sup
\Bigl\{\frac{||\phi^{(\alpha)}(x)||e^{|x|/h}}{h^\alpha\alpha!} :
\alpha\in\mathbb{N}_0,\ x\in (-r,\infty)  \Bigr\}.
$$
Then
$$
\mathcal{P}_*([-r,\infty]):=\mbox{ind}\lim_{h\rightarrow
+\infty}\mathcal{P}_*([-r,\infty],h).
$$
\begin{lem}\label{dense}
$\mathcal P_*(\mathbb D)$ is dense in $\mathcal P_*([-r,\infty]).$
\end{lem}
\begin{proof}
This is a consequence of  Lemma 8.6.4 in \cite{kan}.
\end{proof}
For $a\geq 0,$ we define the space
$$
\mathcal{P}_{*,a}([-r,\infty]):=\{\phi  : e^{a\cdot}\phi \in
\mathcal{P}_*([-r,\infty])\}.
$$
The topology of $\mathcal{P}_{*,a}([-r,\infty])$ is defined by:
$$
\lim \limits_{n\rightarrow \infty}\phi_{n}=0\mbox{ in
}\mathcal{P}_{*,a}([-r,\infty])\mbox{ iff }\lim
\limits_{n\rightarrow \infty}e^{a \cdot}\phi_{n}=0 \mbox{ in
}\mathcal{P}_*([-r,\infty]).
$$
If $a\geq 0$ and $e^{-a \cdot}G \in {\mathcal
Q}_{+}(\mathbb{D},L(E)),$ then $G$ can be extended to an element of
the space of continuous linear mappings from
$\mathcal{P}_{*,a}([-r,\infty])$ into $L(E)$ equipped with the
strong topology. This extension is unique because of Lemma \ref{dense}. We will use
this for the definition of the Laplace
transform of $G.$

\section{Fourier hyperfunction semigroups}
The definition of (exponential) Fourier
hyperfunction semigroup with densely defined
infinitesimal generators of Y. Ito (see ~\cite[Definition
2.1]{ito2}) is given on the basis of the space
$\mathcal{P}_0$ whose structure is not
clear to authors. Our definition is different
and related to non-densely defined infinitesimal generators.

In the sequel, we use the notation ${\mathcal
Q}_{+}(\mathbb{D},L(E))$ for the space of vector-valued Fourier
hyperfunctions supported by $[0,\infty].$ More precisely, if
$f\in {\mathcal Q}_{+}(\mathbb{D},L(E))$ is represented by
$f(t,\cdot)=F_{+}(t+i0,\cdot)-F_{-}(t-i0,\cdot),$ where $F_{+}$ and
$F_{-}$ are defining functions for $f$ (see ~\cite[Definition 1.3.6,
Definition 8.3.1]{kan}) and $\gamma_{+}$ and $\gamma_{-}$ are
piecewise smooth paths connecting points $-a$  ($a>0$) and $\infty$ such that
$\gamma_{+}$ and $\gamma_{-}$ lie respectively in the upper and the
lower half planes as well as in a strip around ${\mathbb R} $
depending on $f,$ then for any $\psi \in {\mathcal
P}_{\ast}(\mathbb{D}),$
$$
\int \limits_{\mathbb R}f(t)\psi(t)\, dt=\int
\limits^{\infty}_{0}f(t)\psi(t)\, dt:=\int
\limits_{\gamma_{+}}F_{+}(z)\psi(z)\, dz-\int
\limits_{\gamma_{-}}F_{-}(z)\psi(z)\, dz.
$$
Since we will use the duality approach of Chong and Kim, we will use notation $\langle f, \psi\rangle$
for the above expression.

Let $\varphi\in\mathcal P_*$  and let
$f(t,\cdot)=F_{+}(t+i0,\cdot)-F_{-}(t-i0,\cdot)$
be an element in ${\mathcal Q}_{+}(\mathbb{D},L(E))$.
Then $$\varphi(t)f(t,\cdot):=\varphi(t)
F_{+}(t+i0,\cdot)-\varphi(t)F_{-}(t-i0,\cdot).$$
\indent We will denote by $\mathcal{P}_{*}^{0}$
a subspace of $\mathcal{P}_{*}$
consisting of functions
$\phi$ with the property $\phi(0)=0.$
Also, we will  consider
$\mathcal{P}_{*}^{00}$, a subspace of
$\mathcal{P}_{*}$ consisting of functions $\psi$
with the properties $\psi(0)=0$ and
$\psi'(0)=0.$ Note, any $\psi \in \mathcal{P}_{*}$
can be written in the form
\begin{equation} \label{hipravan}
\psi(t)=\psi(0)\phi_0(t)+\theta(t),\; t\in \mathbb{R},\, \mbox{respectively}\, ,
\end{equation}
\begin{equation} \label{hiphipravan}
\psi(t)=\psi(0)\phi_0(t)+\psi'(0)\phi_1(t)+
\tilde{\theta}(t),\; t\in \mathbb{R},
\end{equation}
where $\phi_0$ and $\phi_1$ are fixed elements of
$\mathcal{P}_{*}$ with the properties $\phi_0(0)=1,$
$\phi'_0(0)=0,$ $\phi_1(0)=0$, $\phi'_1(0)=1$
and $\theta$ varies over $\mathcal{P}_{*}^{0}$  respectively $\tilde{\theta}$
varies over $\mathcal{P}_{*}^{00}$.
We define $\mathcal{P}_{_*a}^0$ as a space of functions
$\phi\in\mathcal{P}_{_*,a}$ with the property $\phi(0)=0$
and $\mathcal{P}_{_*,a}^{00}$, as a space of functions
$\phi\in\mathcal{P}_{_*,a}$ with the property
$\phi(0)=0, \phi'(0)=0$
and note that the similar decompositions as
(\ref{hipravan}) and (\ref{hiphipravan})
hold for elements of $\mathcal{P}_{_*,a}^0$ and
$\mathcal{P}_{_*,a}^{00}$, respectively.
\begin{defn}\label{hipi}
An element  $G\in {\mathcal Q}_{+}(\mathbb{D},L(E))$ is called  a
pre-Fourier hyperfunction semigroup, if the next condition is valid

(H.1) $ G(\phi*_0\psi) = G (\phi)G(\psi),\
 \phi,\ \psi \in
{\mathcal P}_*(\mathbb{D}).$

Further on, a pre-Fourier hyperfunction semigroup $G$ is called a
Fourier hyperfunction semigroup, (FHSG) in short, if, in addition,
the following holds

(H.2) $ {\mathcal{N}(G)}:= \bigcap_{\phi\in {\mathcal P}
_{*}^{00}(\mathbb{D})}N (G(\phi)) = \{0\}.$

If the next condition also holds:

(H.3) $ {\mathcal R(G)}:= \bigcup_{\phi \in {\mathcal
P}_{*}^{00}(\mathbb{D})}R(G(\phi))$ is dense in $E,$ then $G$ is called a
dense (FHSG).

If $e^{-a \cdot}G \in {\mathcal
Q}_{+}(\mathbb{D},L(E)),\mbox{ for some }\; a >0,$ and (H.1) holds with
$ \phi,\ \psi \in
{\mathcal P}_{_*,a}(\mathbb{D})$ then $G$ is called
exponentially bounded pre-Fourier hyperfunction semigroup.
If (H.2) and (H.3) hold with
$\phi\in {\mathcal P}
_{_*,a}^{00}(\mathbb{D}),$
 then
$G$ is called a dense exponential Fourier hyperfunction semigroup,
dense (EFHSG), in short.
\end{defn}

Let $A$ be a closed operator. We denote by $[D(A)]$ the
Banach space $D(A)$ endowed with the graph norm $\|x\|_{[D(A)]}=\|x\|+\|Ax\|$, $x\in D(A)$.
Like in ~\cite[Definition 2.1,
Definition 3.1]{ito1}, we give the following definitions:

\begin{defn}
Let $A$ be a closed operator. Then $G\in  {\mathcal
Q}_{+}(\mathbb{D},L(E,[D(A)]))$ is a Fourier hyperfunction solution
for $A$ if $P*G=\delta\otimes I_E$ and $G*P=\delta\otimes
I_{[D(A)]}$, where
 $P:=\delta' \otimes
I_{D(A)}-\delta \otimes A \in {\mathcal
Q}_{+}(\mathbb{D},L([D(A)],E));$ $G$ is called an exponential
Fourier hyperfunction solution for $A$ if, additionally,
$$e^{-a \cdot}G \in {\mathcal Q}_{+}(\mathbb{D},L(E,[D(A)])),\;
\mbox{ for some } a> 0.$$ Similarly, if $G$ is an exponential
Fourier hyperfunction solution for $A$ which fulfills (H.3), then
$G$ is called a dense, exponential Fourier hyperfunction solution
for $A.$
\end{defn}
Let $a\geq 0$ and  $\alpha\in {\mathcal{P}_{_,*a}}$,
be an even function such that $\int\alpha(t)\, dt=1.$
Let $\mbox{sgn }(x):=1, x>0, \;\mbox{sgn }(x):=-1,x<0$
and $\mbox{sgn }(0):=0$.
A net of the form $\delta_\varepsilon=
\alpha(\cdot/\varepsilon)/{\varepsilon},
\varepsilon \in (0,1),$
is called  delta net in ${\mathcal{P}_{_,*a}}$.
 Changing $\alpha$ with the above properties,
 one obtains a set of delta nets in ${\mathcal{P}_{_,*a}}$.
 Clearly, every delta net converges to $\delta$
 as $\varepsilon \rightarrow 0$ in
 ${\mathcal Q}(\mathbb{D})$.
We define, for $x\in \mathbb{R},$
$$\delta*_0\phi(x):=2\mbox{sgn }(x)
\lim_{\varepsilon\rightarrow 0}
\delta_\varepsilon*_0\phi(x)=\phi(x),
\; \phi \in{\mathcal{P}^{0}_{_*,a}},
$$
$$\delta'*_0\phi(x):=2\mbox{sgn }(x)
\lim_{\varepsilon\rightarrow 0}\delta'_\varepsilon*_0\phi(x)=\phi'(x),
\; \phi \in{\mathcal{P}^{00}_{_*,a}}.
$$
\begin{defn} \label{infgh}
 Let $a\geq 0$ and $G$ be an (EFHSG). Then

1. $G(\delta)x:=y$ iff $G(\delta*_0\phi)x=G(\phi)y$ for every $\phi\in
{\mathcal{P}}^0_{_*,a}({\mathbb D})$.

2. $G(-\delta')x:=y$ if
$G(-\delta'*_0\phi)x=G(\phi)y$ for every
$\phi\in {\mathcal{P}}^{00}_{_*,a}({\mathbb D})$.\\
$A=G(-\delta')$ is called the infinitesimal generator of $G.$
\end{defn}

Thus $G(\delta)$ is the identity operator.
In order to prove that $G(-\delta')$ is a single-valued function, we have
 to prove that for every $x\in E$, $G(-\delta')x=y_1 $ and
$G(-\delta')x=y_2$ imply $y_1=y_2.$ This means that we have to prove that
$$G(\phi')x=G(\phi)y_1,\; G(\phi')x=
G(\phi)y_2, \;\phi \in {\mathcal{P}}_{*}^{00} \Longrightarrow
y_1=y_2.$$
\begin{prop} \label{good}
If $G(\phi')x=0$  for every $\phi\in \mathcal{P}_{_*,a}^{00},$
then $x=0.$
\end{prop}
\begin{proof}
We shall prove that the assumption  $G(\phi)y=0$ for every $\phi\in {\mathcal{P}}_{_*,a}^0 $ implies that $y=0.$
By (\ref{hipravan}), we have that for any $\phi_0 \in \mathcal{P}_{_*,a}$
such that $\phi_0(0)=c\neq 0$   $$G(\psi)y=\frac{\psi(0)}{c}G(\phi_0)y, \psi \in \mathcal{P}_{_*,a}.$$
Now let $\phi, \psi$ be arbitrary elements of
$\mathcal{P}_{_*,a}.$ Since  $G(\phi*_0\psi)y=G(\phi)G(\psi)y$ and  $\phi*_0\psi(0)=0,$
it follows, with $z=G(\psi)y,$
$$G(\phi*_0\psi)y=G(\phi)z=0,\; \phi \in \mathcal{P}_{_*,a} \Longrightarrow z=0.
$$
Thus, for any $\psi\in \mathcal{P}_{_*,a},$ we have $G(\psi)y=0$ which finally implies $y=0.$

Now, we will prove the assertion.
By (\ref{hiphipravan}) we have that for every $\psi\in \mathcal{P}_{_*,a}$
$$G(\psi')x=\psi(0)G(\phi'_0)x+\psi'(0)G(\phi'_1)x=0.
$$
Denote by $P_{10}$ the set of all $\phi_0\in \mathcal{P}_*$ with the properties
$\phi_1(0)=c\neq 0, \phi'_1(0)=0$
and by $P_{01}$ the set of all $\phi_1\in \mathcal{P}_*$ with the properties
$\phi_0(0)=0, \phi'_0(0)=c\neq 0$.

We have the following cases:
$$
(\forall\phi_0 \in P_{10})(\forall
\phi_1\in P_{01}) (G(\phi_0)x=0,\; G(\phi_1)x=0); $$
$$
(\forall\phi_0 \in P_{10})(\exists
\phi_1\in P_{01}) (G(\phi_0)x=0,\; G(\phi_1)x\neq 0);
$$
$$
(\exists \phi_0 \in P_{10})(\forall
\phi_1\in P_{01}) (G(\phi_0)x\neq 0,\; G(\phi_1)x=0);
$$
$$
(\exists\phi_0 \in P_{10})(\exists
\phi_1\in P_{01}) (G(\phi_0)x\neq 0,\; G(\phi_1)x\neq 0).
$$

In the first case we have, by (\ref{hiphipravan}),
 $G(-\psi')x=0, \psi \in{\mathcal{P}}_{_*,a}.$ This implies, by the standard arguments, that
$G(\psi)x=C\int_{\mathbb{R}} \psi(t)\, dt\, x=0, \psi\in {\mathcal{P}}_{_*,a}$ and this holds for $C=0.$
Consider the fourth case.
In this case we have that
$$G(\psi')x=C_1\langle \delta,\psi\rangle x+ C_2\langle \delta',\psi\rangle x$$
and thus,
$$G(\psi')x=C_1\langle \delta,\psi\rangle x+ C_2\langle \delta',\psi\rangle x + C_3\langle {\bf 1},\psi\rangle x,$$
where $\langle {\bf 1},\psi\rangle x=\int_{\mathbb R}\psi(t)\, dt\, x$.
Now, by the semigroup property it follows $C_1=C_2=C_3=0$ and with this we conclude as above that
$x=0$.
We can handle out the second and the third case in a similar way. This completes the proof of the assertion.
\end{proof}

\subsection{Laplace transform and the characterizations of Fourier hyperfunction semigroups}

The proofs of assertions of this section related to the Laplace transform are new but some of them are quite simple.
They are based on the technics developed by Komatsu \cite{k91}-\cite{k92}

Note, for every $r>0,$ $E_\lambda=e^{-\lambda
\cdot}\in\mathcal{P}_*((-r,\infty])$, for every $\lambda\in
\mathbb{C}$ with $ Re \lambda >0.$ So, we  can define the Laplace
transform of $G \in {\mathcal Q}_{+}(\mathbb{D},L(E))$ by
$$\mathcal{L}G(\lambda)=\hat{G}(\lambda):=G(E_\lambda),\ Re\lambda>0.$$
\begin{prop}
There exists a suitable local operator $P$ such that
$$ |\hat{G}(\lambda)|\leq |P(\lambda)|,\; Re \lambda >0.$$
\end{prop}
The proof of this assertion it is even simpler than the proof of the
corresponding assertion in the case of Roumieu ultradistributions.

If $e^{-a \cdot}G \in {\mathcal Q}_{+}(\mathbb{D},L(E)),$ we  define the  Laplace
transform of $G$  by
$${\mathcal L}(G)(\lambda)=\hat{G}(\lambda):={G}(E_{\lambda}),\; Re \lambda>a.$$
It is an analytic function defined on $\{ \lambda \in {\mathbb C} :
Re \lambda>a \}$ and there exists a local operator $P$ such that
$|{\hat G}(\lambda)|\leq |P(\lambda)|, Re \lambda>a.$

\begin{rem} Similarly to the corresponding Roumieu case, one can prove the
next statement: \\If $G\in {\mathcal Q}_{+}(\mathbb{D},L(E,[D(A)]))$
is a Fourier hyperfunction solution for $A,$ then $G$ is a
pre-Fourier hyperfunction semigroup. It can be seen, as in the case
of ultradistributions, that we do not have that $G$ must be an
(FHSG).\end{rem}

Structural properties of the Fourier hyperfunction semigroups are similar to
that of ultradistribution semigroups of Roumieu class.
For the essentially different proofs of corresponding results we need
the next lemma where we again use the Fourier transform instead of Laplace transform.

\begin{lem} \label{izoh}
Let $P_{L_p}$ be of the form (\ref{str3}). The mapping
$$P_{L_p}(id/dt): {\mathcal P}_*(\mathbb{D})\rightarrow {\mathcal P}_*(\mathbb{D}),\;\;
\phi\mapsto P_{L_p}(id/dt)\phi
$$
is a continuous linear bijection.
\end{lem}

\begin{proof}
Due to ~\cite[Proposition 8.2.2]{kan}, $\phi\in {\mathcal
P}_*(\mathbb{D})$ implies ${\mathcal F}(\phi)\in {\mathcal
P}_*(\mathbb{D})$. Thus, for some $n\in {\mathbb N},$ every
$\varepsilon>0$ and a corresponding $C_\varepsilon>0,$ $|{\mathcal
F}(\phi)(z)|\leq C_\varepsilon e^{(-1/n-\varepsilon)|Re z|},$ $ z\in
{\mathbb R}+I_n.$ By ~\cite[Proposition 8.1.6, Lemma 8.1.7, Theorem
8.4.9]{kan}, with some simple modifications, we have
\begin{equation}\label{nejhi}
 Ce^{\frac{A|\zeta|}{r(|\zeta|+1)}}\leq |P_{L_p}
 (\zeta)|, \;|\eta|\leq\frac{|\xi|}{2}+\frac{1}{L_1},\; \zeta=\xi +i \eta,
\end{equation}
for some $C,\ A>0$ and some monotone increasing function $r$ with
the properties $r(0)=1,\; r(\infty)=\infty.$ This implies that there
exists an integer $n_0\in{\mathbb N}$ such that
$$
{\mathcal F}(\phi)/P_{L_p}\in \tilde{\mathcal O}^{-1/n_0}({\mathbb
R}+iI_{n_0}).
$$
Thus, its inverse Fourier transform ${\mathcal F}^{-1}({\mathcal
F}(\phi)/P_{L_p})$ is an element of ${\mathcal P}_*(\mathbb{D}).$
\end{proof}

Using the properties of local operators as well as norms
$||\cdot||_{h,p!}$, as in the case of Roumieu  tempered
ultradistributions, one obtains the following
assertions.

\begin{thm}\label{fhlap1}
Suppose that  $f:\{\lambda\in {\mathbb C}:\ Re
\lambda>a\}\rightarrow E$ is an analytic function satisfying
$$
||f(\lambda)||\leq C|P(\lambda)|,\; Re \lambda>a,
$$ for some  $C>0,$
some local operator $P$  with the property $|P(\lambda)|>0,\; Re
\lambda>a.$ Suppose, further, that a local operator $\tilde{P}$
satisfies (\ref{nejhi}). Then
$$(\exists M>0) (\exists h\in
C^\infty([0,\infty);E))(\forall j\in\mathbb{N}_0)(h^{(j)}(0)=0)$$
such that
$||h(t)||\leq Me^{a t},\; t\geq 0 , \mbox{ and }$
$$
f(\lambda)=P(\lambda)\tilde{P}(\lambda)\int_0^\infty e^{-\lambda
t}h(t)\, dt,\; Re \lambda>a.
$$
\end{thm}

\begin{thm}\label{fhesg}
Let $A$ be closed and densely defined. Then $A$ generates a dense
$(EFHSG)$ iff the following conditions are true:
\begin{itemize}
\item[(i)] $\{ \lambda \in {\mathbb C} : Re\lambda>a\} \subset \rho(A).$
\item[(ii)] There exist a local operator $P$  with the property
$|P(\lambda)|>0,\; Re \lambda>a$, a local operator
$\tilde{P}$ with the properties as in the previous theorem
 and $C>0$ such that
$$
||R(\lambda : A)||\leq C |P(\lambda)\tilde{P}(\lambda)|,\; Re \lambda >a.
$$
\item[(iii)] $R(\lambda : A)$ is the Laplace transform of some $G$ which satisfies
$(H.2).$
\end{itemize}
\end{thm}
\begin{proof} We will prove the theorem for $a=0.$
 $(\Leftarrow)$:  Theorem \ref{fhlap1} implies
that $R(\lambda:A)$ is of the form
$$R(\lambda : A)= P(\lambda)\tilde{P}
(\lambda)\int_0^\infty e^{-\lambda t}S(t)\, dt,\ Re \lambda>0,
$$
where $S\in C^\infty([0,\infty))$, $S^{(j)}(0)=0,
\;j\in\mathbb{N}_0$ and for every $\varepsilon
>0$ there exists $M>0$ such that $||S(t)||\leq M ,\ t\geq 0$ This
implies $R(\lambda : A)={\mathcal L}(G)(\lambda),\; Re \lambda>0,$
where $G=P(-d/dt)\tilde{P}(-d/dt) S,\;$ and $G\in {\mathcal
Q}_+({\mathbb D},E).$ Since
$$
(\delta'\otimes I_{D(A)}-\delta\otimes A)*G=\delta \otimes I_E,
$$
$$G*(\delta'\otimes I_{D(A)}-\delta \otimes A)=\delta \otimes
I_{D(A)},
$$
and (iii) holds, we have that  $G$ is a Fourier hyperfunction
semigroup.

$(\Rightarrow)$: Put $E_\lambda^+=E_\lambda H, R_\lambda^+=R_\lambda H,$ where $H$ is Heaviside's function. Let $G\in{\mathcal Q}_+({\mathbb D},L(E,D(A)))$
and $\lambda\in\{z\in{\mathbb C} : Re\lambda>a\}\subset\rho(A)$ be
fixed. Then $(\delta'+\lambda\delta)\ast E_{\lambda}^+=\delta.$
Now let $\phi\in{\mathcal P}_{\ast}({\mathbb D})$ and $x\in E.$ Then
$$G((\delta'+\lambda\delta)\ast_0 E_{\lambda}^+\ast_0\phi)=G(\phi)x,$$
and $$G(\delta'\ast_0R_{\lambda}^+\ast_0\phi)x+\lambda
G(\delta\ast_0E_{\lambda}^+\ast_0\phi)x=G(\delta')
G(E_{\lambda}^+\ast_0\phi)x+\lambda{\hat{G}}(\lambda)G(\phi)x\, .$$ Hence,
$$-A({\hat{G}}(\lambda)G(\phi)x)+\lambda\hat{G}(\lambda)
G(\phi)x=G(\phi)x\, .$$ Since $(H.3)$ is assumed
$(-A+\lambda){\hat{G}}(\lambda)=I$, so
$\|\hat{G}(\lambda)\|\leq C|P(\lambda)|$, $Re\lambda>a$,
where $P$ is an appropriate local operator.
\end{proof}

\begin{cor}\label{fhnondesg}
Suppose $A$ is a closed linear operator.
If $A$ generates an (EFHSG), (i), (ii) and (iii) of Theorem
\ref{fhesg} hold.

If (i) and (ii) of
Theorem \ref{fhesg} hold, then $G,$ defined in the same way as above, is a
Fourier hyperfunction fundamental solution for $A$. If
(iii) is satisfied, then $G$ is an (EFHSG) generated by $A$.
\end{cor}
We note that in Corollary \ref{fhnondesg} the operator $A$ is non--densely defined.\\
\indent Now we will prove a theorem related to Fourier hyperfunction
semigroups. As in the case of ultradistributions, the theorem can be
proved for (EFHSG) but for the sake of simplicity,  we will assume
that $a=0.$

We need one more theorem.
\begin{thm}\label{hip6.015}
Let $A$ be a closed operator in $E$. If $A$ generates a (FHSG) $G$,
then $G$ is an Fourier hyperfunction  fundamental solution for
$$P:=\delta' \otimes I_{D(A)}-\delta \otimes A \in {\mathcal
Q}_{+}(\mathbb{D},L([D(A)],E)).$$ In particular, if $T\in {\mathcal
Q}_+(\mathbb{D},E)$, then $u=G*T$ is the unique solution of
\begin{equation} \label{jedn}
-Au+\frac{\partial}{\partial t}u=T,\; u\in{\mathcal
Q}_+(\mathbb{D},[D(A)]).
\end{equation}
If \text{supp}$T\subset[\alpha,\infty),$ then
\text{supp}$u\subset[\alpha,\infty).$

Conversely, if $G \in {\mathcal Q}_{+}(\mathbb{D},L(E,[D(A)]))$ is
a Fourier hyperfunction fundamental solution for $P$ and ${\mathcal
N}(G)=\{0\},$ then $G$ is an (FHSG) in $E$.
\end{thm}
\begin{proof}
$(\Rightarrow)$ One can simply check that $(G(\psi )x,G(-\psi
^{\prime })x-\psi (0)x)\in G(-\delta ^{\prime })$ and $G$ is a
fundamental solution for $P$. The uniqueness of the solution $u=G*T$
of (\ref{jedn}) is clear as well as the support property for the
solution $u$ if supp$T\subset[\alpha,\infty)$.

The part $(\Leftarrow )$ can be proved in the same way as in the
 ~\cite[Theorem 3.3]{msd}, part (d)$\Rightarrow$ (a).
\end{proof}

  First, we list the statements:

  \vspace{0,5cm}

\begin{itemize}
\item[(1)]
$A$ generates an (FHSG) $G.$
\item[(2)] $A$ generates an (FHSG) of the
form $G=P_{L_p}(-id/dt)S_{a,K}$,  where $S_{K}:{\mathbb
R}\rightarrow L(E)$ is exponentially slowly increasing continuous
function and $S_{K}(t)=0,\; t\leq 0.$
\item[(3)]  $A$ is the generator  of a global $K$-convoluted  semigroup
$(S_{K}(t))_{t\geq  0}$ , where $K={\mathcal
L}^{-1}(\frac{1}{P_{L_p}(-i\lambda)}).$
\item[(4)] The problem
$$(\delta\otimes(-A)+\delta'\otimes I_E)*G=\delta \otimes I_E,
\;\; G*(\delta\otimes(-A)+\delta'\otimes I_{D(A)})=\delta \otimes
I_{D(A)}
$$
has a unique solution $G \in {\mathcal Q}_+(\mathbb{D},L(E,[D(A)]))$
with ${\mathcal N}(G)=\{0\}$.
\item[(5)] For every $\varepsilon>0$ there exists $K_\varepsilon>0$
such that
$$\rho(A)\supset \{\lambda\in {\mathbb{C}} :\mbox{   } Re \lambda
>0\}
$$ and
$$ ||R(\lambda : A)||\leq K_\varepsilon e^{\varepsilon|\lambda|},\; Re
\lambda> 0.
$$
\end{itemize}


\begin{thm}\label{hypse}
 (1) $\Leftrightarrow$ (4);
(1) $\Rightarrow$ (3); (3) $\Rightarrow$ (4); (4) $\Rightarrow$ (5);
\end{thm}
\begin{proof}
The equivalence of (1) and (4) can be proved in the same way as in
the case of ultradistribution semigroups, ~\cite[Theorem 3.3]{msd}.

One must use Lemma \ref{izoh} in proving of (1) $\Rightarrow$ (3)
(see ~\cite[Theorem 3.3 ]{msd}(a)' $\Rightarrow$ (c)').
The implication (4) $\Rightarrow$ (5) is a
consequence of Theorem \ref{fhesg} and Corollary \ref{fhnondesg}. In
the case when the infinitesimal generator is densely defined Y. Ito
\cite{ito1} proved the equivalence of a slightly different assertion
$(4)$,
 without the assumption $\mathcal{N}(G)=\{0\},$ and $(5).$
Our assertion is the stronger one since it is based on the strong
structural result of Theorem \ref{fhesg}.
\end{proof}
Operators which satisfy $(5)$ may be given using the analysis of P.C. Kunstmann ~\cite[Example 1.6]{ku101} with suitable chosen sequence $(M_p)_{p\in{{\mathbb N}_{0}}}.$\\


\indent The definition of a hyperfunction fundamental solution $G$ for a
closed linear operator $A$ can be found in the paper \cite{o192} of
S. \={O}uchi. For the sake of simplicity, we shall also say, in that
case, that $A$ generates a hyperfunction semigroup $G.$ The next
assertion is proved in \cite{o192}:

A closed linear operator $A$ generates a hyperfunction semigroup iff
for every $\varepsilon>0$ there exist suitable $C_{\varepsilon},\
K_{\varepsilon}>0$ so that
$$\rho(A)\supset \Omega_{\varepsilon}:=\{\lambda\in {\mathbb{C}} :\mbox{   } Re \lambda \geq
\varepsilon|\lambda|+C_\varepsilon\}$$ and
$$ ||R(\lambda:A)||\leq K_\varepsilon e^{\varepsilon|\lambda|},\; \lambda\in \Omega_{\varepsilon}.$$
We will give some results related to hyperfunction and convoluted
semigroups in terms of spectral conditions and the asymptotic
behavior of $\tilde {K}$. We refer to \cite{a22} for the similar
results related to $n$-times integrated semigroups, $n\in {{\mathbb
N}_{0}},$ to \cite{l115} for $\alpha$-times integrated semigroups,
$\alpha>0$ and to ~\cite[Theorem 1.3.1]{me152} for convoluted
semigroups. Since we focus our attention on connections of
convoluted semigroups with hyperfunction semigroups, we use the next
conditions for $K:$
\begin{itemize}
\item[(P1)] $K$ is exponentially bounded, i.e., there exist $\beta \in {\mathbb R}$ and
$M>0$ so that $|K(t)|\leq Me^{\beta t},$ for a.e. $t\geq 0.$
\item[(P2)] $\tilde{K}(\lambda ) \neq 0,\; Re\lambda >\beta. $
\end{itemize}
In general, the second condition does not hold for  exponentially
bounded functions, cf. ~\cite[Theorem 1.11.1]{a43} and \cite{mn}.
Following  analysis in \cite{c62} and ~\cite[Theorem 2.7.1, Theorem 2.7.2]{kosticknjiga},
in our context, we can give the following statements:
\begin{thm}\label{conhy}
\begin{itemize}
\item[1.]
Let $K$ satisfy $(P1)$ and $(P2)$ and let $(S_{K}(t))_{t\in
[0,\tau)},$ $0<\tau \leq \infty$, be a $K$-convoluted semigroup
generated by $A.$
Suppose that for every $\varepsilon>0$ there exist
$\varepsilon_{0}\in (0,\tau \varepsilon)$ and $T_{\varepsilon}>0$
such that
$$
\frac{1}{|\tilde{K}(\lambda)|}\leq
T_{\varepsilon}e^{\varepsilon_{0}|\lambda|},\mbox{   }\lambda \in
\Omega_{\varepsilon} \cap \{\lambda \in {\mathbb
C}:Re\lambda>\beta\}.
$$
Then for every $\varepsilon>0$ there exist
$\overline{C}_{\varepsilon}>0$ and $\overline{K}_{\varepsilon}>0$
such that
$$
\Omega_{\varepsilon}^{1}=\{\lambda \in {\mathbb C} : Re\lambda
\geq \varepsilon |\lambda|+\overline{C}_{\varepsilon}\}\subset
\rho(A)\mbox{  and } ||R(\lambda:A)||\leq
\overline{K}_{\varepsilon}e^{\varepsilon_{0}|\lambda|},\mbox{  }
\lambda \in \Omega_{\varepsilon}^{1}.$$
\item[2.]
Let $K\in L^{1}_{loc}([0,\tau))$ for some $0<\tau \leq 1$ and let
$A$ generate a $K$-convoluted semigroup $(S_{K}(t))_{t\in
[0,\tau)}$. If $K$ can be extended to a function $K_{1}$ in $
L^{1}_{loc}([0,\infty))$ which satisfies (P1) so that its Laplace
transform has the same estimates as in Theorem $\ref{conhy}$, then
$A$ generates S. \={O}uchi's hyperfunction semigroup.
\item[3.]
Assume that for every $\varepsilon>0$ there exist
$C_{\varepsilon}>0$ and $M_{\varepsilon}>0$ so that
$\Omega_{\varepsilon} \subset \rho(A)$ and that
$||R(\lambda:A)||\leq M_{\varepsilon}e^{\varepsilon
|\lambda|},\mbox{    }\lambda \in \Omega_{\varepsilon}.$
\begin{itemize}
\item[(a)] Assume that $K$ is an exponentially bounded function with
the following property for its Laplace transform: There exists
$\varepsilon_0>0$ such that for every $\varepsilon>0$ exists
$T_{\varepsilon}>0$ with
\begin{equation}\label{str4}
|\tilde{K}(\lambda)|\leq
T_{\varepsilon}e^{-\varepsilon_{0}|\lambda|},\mbox{    }\lambda
\in \Omega_{\varepsilon}.
\end{equation}
If
$\tau>0$ and $K_{|[0,\tau)}\neq 0$ ($K_{|[0,\tau)}$ is the
restriction of $K$ on $[0,\tau)$, then $A$ generates a local
$K$-semigroup on $[0,\tau)$.
\item[(b)] Assume that $K$ is an exponentially bounded function,
$\tau>0$ and $K_{|[0,\tau)}\neq 0.$
Assume that for every $\varepsilon>0$ there exist
$T_{\varepsilon}>0$ and $\varepsilon_{0}\in (\varepsilon (1+\tau),\
\infty)$ such that $(\ref{str4})$ holds. Then $A$ generates a local
$K$-semigroup on $[0,\tau).$
\end{itemize}
\end{itemize}
\end{thm}

Connections of hyperfunction and ultradistribution semigroups with
(local integrated) regularized semigroups seems to be more
complicated.
In this context, there is a example (essentially due to R. Beals \cite{b41})
which shows that there exists a densely
defined operator $A$ on the Hardy space $H^{2}({\mathbb C}_{+})$
which has the following properties:

1. $A$ is the generator of S. \={O}uchi's hyperfunction semigroup.

2. $A$ is not a subgenerator of a local $\alpha$-times integrated
$C$-semigroup, for any injective $C\in L(H^{2}({\mathbb C}_{+}))$
and $\alpha>0.$

It is clear that there exists an operator $A$ which generates an
entire $C$-regularized group but not a hyperfunction semigroup.

\vspace{\baselineskip}


\begin{thebibliography}{99}

\bibitem{a11} {W. Arendt},  \textit{Vector-valued Laplace transforms and Cauchy problems.}
Israel J. Math. \textbf{59} (1987), 327--352.

\bibitem{a22} {W. Arendt, O. El-Mennaoui\and V. Keyantuo},  \textit{Local integrated
semigroups: evolution with jumps of regularity.} J. Math. Anal.
Appl. \textbf{186} (1994), 572--595.

\bibitem{a43} {W. Arendt, C. J. K. Batty, M. Hieber \and F. Neubrander},
\textit{Vector-valued Laplace Transforms and Cauchy Problems.}
Birkh\" auser Verlag, 2001.


\bibitem{b41} {R. Beals},  \textit{On the abstract Cauchy problem.} J. Funct. Anal.
\textbf{10} (1972), 281--299.

\bibitem{b42} {R. Beals},  \textit{Semigroups and abstract Gevrey spaces.} J. Funct. Anal.
\textbf{10} (1972), 300-308.



\bibitem{cha} {J. Chazarain},  \textit{Probl\' emes de Cauchy abstraites et applications \' a
quelques probl\' emes mixtes.} J. Funct. Anal. \textbf{7} (1971),
386--446.

\bibitem{chungovi} {J. Chung, S.-Y. Chung \and D. Kim},
\textit{Characterization of the Gelfand-Shilov spaces via Fourier
transforms.} Proc. of AMS \textbf{124} (1996), 2101--2108.

\bibitem{chungov} {J. Chung, S.-Y. Chung \and D. Kim}, \textit{A
characterization for Fourier hyperfunctions.} Publ. Res. Inst. Math.
Sci. \textbf{30} (1994), 203--208.

\bibitem{ci1} {I. Cior\u anescu},  \textit{Beurling spaces of class $(M\sb{p})$ and
ultradistribution semi-groups.} Bull. Sci. Math. \textbf{102}
(1978), 167--192.

\bibitem{c62} {I. Cior\u anescu, G. Lumer},  \textit{Probl\` emes d'\' evolution r\' egularis\'
es par un noyan g\' en\' eral $K(t)$. Formule de Duhamel,
prolongements, th\' eor\` emes de g\' en\' eration.} C. R. Acad.
Sci. Paris S\' er. I Math. \textbf{319} (1995), 1273--1278.



\bibitem{l1} {R. deLaubenfels},  \textit{Existence Families, Functional Calculi and
Evolution Equations.} Lect. Notes Math. \textbf{1570}, Springer
1994.

\bibitem{er} {H. A. Emami-Rad},  \textit{Les semi-groupes distributions de Beurling.}
C. R. Acad. Sc. S\' er. A  \textbf{276} (1973), 117--119.

\bibitem{fat} {H. O. Fattorini},  \textit{Structural theorems for
vector valued ultradistributions.} J. Funct. Anal. \textbf{39}
(1980), 381-407.


\bibitem{hor} {L. H\" ormander},  \textit{Between distributions
and hyperfunctions.} Colloq. Honneur L. Schwartz, Ec. Polytech.
1983, Vol 1,  Ast\' erisque \textbf{131} (1985), 89-106.

\bibitem{l115} {M. Hieber},
\textit{Integrated semigroups and differential
operators on $L^{p} $ spaces.} Math. Ann, \textbf{29} (1991), 1- 16.

\bibitem{ito1} {Y. Ito}, \textit{On the abstract Cauchy problems in the sense
of Fourier hyperfunctions.} J. Math. Tokushima Univ. \textbf{16}
(1982), 25-31.

\bibitem{ito2} {Y. Ito}, \textit{Fourier hyperfunction semigroups.}
J. Math. Tokushima Univ. \textbf{16} (1982), 33-53.


\bibitem{kan} {A. Kaneko}, \textit{Introduction to Hyperfunctions.}
Kluwer, Dordercht, Boston, London, 1982.

\bibitem{kaw1} {T. Kawai}, \textit{The theory of Fourier transformations in the theory of hyperfunctions and its
applications.} Surikaiseki Kenkyusho Kokyuroku, R. I. M. S., Kyoto
Univ., \textbf{108} (1969), 84--288 (in Japanese).

\bibitem{kaw2} {T. Kawai}, \textit{On the theory of Fourier
hyperfunctions and its applications to partial differential
equations with constant coefficients.} J. Fac. Sci., Univ. Tokyo,
Sec. IA. \textbf{17} (1970), 465--517.

\bibitem{keya} {V. Keyantuo}, \textit{Integrated semigroups and related partial differential
equations.} J. Math. Anal. Appl. \textbf{212} (1997), 135--153.

\bibitem{ki90} {J. Kisy\'{n}ski},  \textit{Distribution semigroups and
one parameter semigroups.} Bull. Polish Acad. Sci. \textbf{50}
(2002), 189--216.


\bibitem{kochu} {A.N. Kochubei}, \textit{Hyperfunction solutions of differential--operator equations.}
Siberian Math. J., \textbf{20}, No. 4 (1979), 544-554.


\bibitem{k91} {H. Komatsu},  \textit{Ultradistributions, I. Structure theorems and a
characterization.} J. Fac. Sci. Univ. Tokyo Sect. IA Math.
\textbf{20} (1973), 25--105.

\bibitem{k78}{H. Komatsu}, \textit{An introduction to the theory of generalized functions}, Iwanami Shoten, 1978,
translatated by L. S. Hahn in 1984., Department of Mathematics Sciences University of Tokyo.

\bibitem{k82} {H. Komatsu},  \textit{Ultradistributions, III. Vector valued
ultradistributions the theory of kernels.} J. Fac. Sci. Univ. Tokyo
Sect. IA Math. \textbf{29} (1982), 653--718.

\bibitem{k92} {H. Komatsu},  \textit{Operational calculus and semi-groups of operators.}
Functional Analysis and Related topics (Kioto), Springer, Berlin,
213-234, 1991.

\bibitem{ko98} {M. Kosti\' c},  \textit{C-Distribution semigroups.} Studia
Math. \textbf{185} (2008), 201--217.

\bibitem{kosticknjiga}  {M. Kosti\' c}, \textit{Generalized semigroups and cosine functions.}
Mathematical Institute,
Belgrade, 2011.

\bibitem{koj} {M. Kosti\' c},  \textit{Convoluted $C$-cosine functions and convoluted
$C$-semigroups.} Bull. Cl. Sci. Math. Nat. Sci. Math. \textbf{28}
(2003), 75--92.

\bibitem{mn} {M. Kosti\' c \and S. Pilipovi\' c},  \textit{Global convoluted semigroups.}
Math. Nachr., \textbf{280}, No. 15 (2007), 1727--1743.

\bibitem{kps} {M. Kosti\' c \and S. Pilipovi\' c},  \textit{Ultradistribution semigroups.} Siberian Math. J.,
\textbf{53}, No. 2 (2012), 232-242.

\bibitem{msd} {M. Kosti\'c, S. Pilipovi\'c \and D. Velinov},
\textit{Structural theorems for ultradistribution semigroups}
accepted in Siberina Math. J.

\bibitem{ku101} {P. C. Kunstmann},  \textit{Stationary dense operators and generation of
non-dense distribution semigroups.} J. Operator Theory \textbf{37}
(1997), 111--120.

\bibitem{ku112} {P. C. Kunstmann},  \textit{Distribution semigroups and abstract Cauchy
problems.} Trans. Amer. Math. Soc. \textbf{351} (1999), 837--856.

\bibitem{ku113} {P. C. Kunstmann},  \textit{Banach space valued ultradistributions and
applications to abstract Cauchy problems.}, http://www.math.kit.edu/iana1/~kunstmann/media/ultra-appl.pdf, preprint.

\bibitem{l114} {M. Li, F. Huang \and Q. Zheng},  \textit{Local integrated $C$-semigroups.}
Studia Math. \textbf{145} (2001), 265--280.

\bibitem{li121} {J. L. Lions},  \textit{Semi-groupes distributions.} Portugal. Math.
\textbf{19} (1960), 141-164.

\bibitem{luno} {G. Lumer \and F. Neubrander}, \textit{The asymptotic Laplace transform: new
results and relation to Komatsu's Laplace transform of
hyperfunctions.}, Partial Differential Equations on Multistructures
(Luminy, 1999), 147--162. Lect. Not. Pure Appl. Math., \textbf{219},
Dekker, New York, 2001.

\bibitem{me152} {I. V. Melnikova \and A. I. Filinkov},  \textit{Abstract Cauchy Problems:
Three Approaches.} Chapman \& Hall/CRC, 2001.

\bibitem{mor}
{M. Morimoto} \textit{An introduction to Sato's hyperfunctions.}
Translations of Mathematical Monographs, 129. American Mathematical Society,
Providence, RI, 1993.

\bibitem{n181} {F. Neubrander},  \textit{Integrated semigroups and their applications to the
abstract Cauchy problem.} Pacific J. Math. \textbf{135} (1988),
111--155.

\bibitem{ler} {Y. Ohya},  \textit{Le probl$\grave{e}$me de Cauchy pour
les \' equations hyperboliques $\grave{a}$ caract\' eristiques
multiples.} J. Math. Soc. Japan \textbf{16} (1964), 268--286.

\bibitem{o192} {S. \={O}uchi},  \textit{Hyperfunction solutions of the abstract Cauchy
problems.} Proc. Japan Acad. \textbf{47} (1971), 541--544.

\bibitem{o193} {S. \={O}uchi},  \textit{On abstract Cauchy
problems in the sense of hyperfunctions in Hyperfunctions
and Pseudo--Differentail Equations}, Proc. Katata 1971,
edited by H. Komatsu, Lect. Notes in Math., \textbf{287},(1973), 135--152.




\bibitem{satovi} {M. Sato}, \textit{Theory of hyperfunctions.} S\^{u}gaku \textbf{10} (1958), 1--27.

\bibitem{w241} {S. Wang},  \textit{Quasi-distribution semigroups and integrated semigroups.}
J. Funct. Anal. \textbf{146} (1997), 352--381.


\end{thebibliography}
\end{document}